\documentclass[11pt,a4paper]{article}
\usepackage{etex}
\usepackage{inputenc}
\usepackage[english]{babel}
\usepackage{lmodern}					
\usepackage{textcomp}				
\usepackage{verbatim}			
\usepackage[toc,page]{appendix} 		
\usepackage{cite}
\usepackage{amsmath,					
			amsthm,
			amsfonts,
			amssymb, 
			mathtools}
\usepackage{stmaryrd}
\usepackage{physics}
\usepackage{enumitem}
\usepackage{hyperref}
\hypersetup{
colorlinks=true,
linkcolor=purple,
filecolor=magenta,
urlcolor=cyan,
citecolor=blue
}

 \usepackage{tikz-cd} 
\usetikzlibrary{tikzmark}
\usetikzlibrary{arrows}
 \usepackage{verbatim}
\definecolor{DarkGreen}{HTML}{1cad22}
\usepackage[hmargin=3cm,vmargin=2.5cm]{geometry}
\usepackage{graphicx} 
\usepackage{amsthm, amssymb, amsmath}
\usepackage{color}
\usepackage{xcolor}
\usepackage{tikz}
\usetikzlibrary{positioning}
\usepackage{tikz-cd}
\usepackage{amsmath}
\usepackage{amssymb}
\usepackage{float}
\usepackage{caption}
\usepackage{mathtools}
\usepackage{hyperref}
\usepackage{mathrsfs}

\usepackage{graphicx}
\usepackage{epsfig}
\numberwithin{equation}{section}

{\theoremstyle{definition}\newtheorem{definition}{Definition}[section]

\newtheorem{defititle}[definition]{\Title}

\newtheorem{remark}[definition]{Remark}

\newtheorem{exs}[definition]{Examples}}
\newtheorem{prop}[definition]{Proposition}
\newtheorem{proposition-definition}[definition]{Proposition-Definition}
\newtheorem{lemma}[definition]{Lemma}
\newtheorem{thm}[definition]{Theorem}
\newtheorem{cor}[definition]{Corollary}

\newtheorem*{theorem}{Theorem}

\newtheoremstyle{named}{}{}{\itshape}{}{\bfseries}{.}{.5em}{\thmnote{#3's }#1}
\theoremstyle{named}

\usepackage{amsthm}

\newenvironment{mainthm-env}[1]
  {\begin{theorem}[Main Theorem #1]}
  {\end{theorem}}

\title{Remarks on Topological Rigidity of Real Moment-Angle Manifolds}

\author{Ioannis Gkeneralis\footnote{Department of Mathematics, Aristotle University of Thessaloniki, \href{mailto:igkeneralis@math.auth.gr}{igkeneralis@math.auth.gr}}}

\date{ }

\newcommand{\keywords}[1]{\textbf{Keywords: } #1}
\newcommand{\msc}[1]{\textbf{Mathematics Subject Classification: } #1}

\begin{document}

\vspace{-0.5cm}

\maketitle

\begin{abstract}
We study topological rigidity of real moment-angle manifolds associated to flag simplicial complexes. Using the cubical geometry arising from the Davis construction, we identify the universal cover with the Davis complex and deduce that it admits a CAT(0) metric. As a consequence, its fundamental group satisfies the Farrell--Jones conjecture. Applying surgery theory, we deduce that real moment-angle manifolds of dimension at least five associated to flag complexes satisfy the Borel Conjecture. We also explain why this rigidity phenomenon is specific to the real case and fails for complex and quaternionic moment-angle complexes.
\end{abstract}

\noindent \keywords{flag complex, asphericity, CAT(0) geometry, Farrell-Jones conjecture, Borel conjecture} \\[1ex]
\msc[2020]{Primary  57S25; Secondary 57R67, 20F67, 55U10, 51F15 }

\vspace{0.5ex}
\noindent \textbf{Conflict of Interest and Data Availability Statement:} The author states that there is no conflict of interest to declare. Moreover, data sharing is not applicable to this article as no datasets were generated or analyzed during the research carried out in this paper.


\section{Introduction}

Moment-angle manifolds arise naturally in toric topology as spaces encoding the combinatorics of simplicial complexes and simple polytopes. Both the complex and real constructions are instances of polyhedral products and exhibit rich combinatorial, equivariant, and cohomological structures. For details see \cite{bbhg,tt} and references therein.

However, their large-scale geometric behavior differs significantly. In the complex case, moment-angle manifolds are closely related to toric varieties and are typically studied using equivariant and cohomological methods \cite{dj}. In contrast, in the real case, additional geometric features emerge under suitable combinatorial assumptions. In particular, when the underlying simplicial complex is flag, real moment-angle manifolds become aspherical and admit geometric structures that place them within the framework of CAT(0) geometry and geometric group theory.

Let $K$ be a finite simplicial complex on $m$ vertices and let
\[
\mathcal{R}_K = Z_K(D^1,S^0)
\]
denote the associated real moment-angle complex (see e.g. \cite{bbhg}). Here $Z_K(D^1,S^0)$ denotes the polyhedral product determined by $K$ and the pair $(D^1,S^0)$, where $D^1=[-1,1]$ is the unit interval and
$S^0=\{\pm 1\}$ is its boundary. Explicitly,
\[
\mathcal{R}_K =
\bigcup_{I \in K}
\left(
\prod_{i \in I} D^1 \times \prod_{j \notin I} S^0
\right),
\]
which realizes $\mathcal{R}_K$ as a cubical subcomplex of $[-1,1]^m$ whose combinatorial structure is determined by $K$. When $K$ is the boundary of a simplicial polytope dual to a simple polytope, $\mathcal{R}_K$ is a closed manifold.

A fundamental result due to Davis \cite{dav, rafca} implies that $\mathcal{R}_K$ is aspherical if and only if $K$ is a flag complex. In this case, the topology of $\mathcal{R}_K$ is governed by the geometry of its fundamental group, placing these manifolds within the scope of the Borel Conjecture. In particular, any homotopy equivalence between such manifolds is homotopic to a homeomorphism. In this note, we prove the following: 

\begin{theorem}[Theorem \ref{thm:Borel_main}]
Let $K$ be a finite flag simplicial complex such that $\mathcal{R}_K$ is a closed manifold of dimension $n \geq 5$. Then $\mathcal{R}_K$ is aspherical and satisfies the Borel Conjecture. In particular, any homotopy equivalence
\[
f \colon N \longrightarrow \mathcal{R}_K
\]
from a closed manifold $N$ is homotopic to a homeomorphism.
\end{theorem}

The proof combines ideas from toric topology, geometric group theory, and surgery theory. The central mechanism is a geometric realization of $\mathcal{R}_K$ as a quotient of a cubical complex arising from the Davis construction.

More precisely, one associates to $K$ a cubical complex whose universal cover is naturally identified with the Davis complex of the associated right-angled Coxeter group $W_K$. The link of each vertex in this complex is isomorphic to $K$. When $K$ is flag, Gromov's link condition \cite{gromov} implies that the universal cover is a CAT(0) space (see also \cite[p.212]{BridsonHaelfiger}). In particular, it is contractible, and $\mathcal{R}_K$ is aspherical.

The fundamental group $\pi_1(\mathcal{R}_K)$ acts properly, cocompactly, and isometrically on this CAT(0) space (Proposition \ref{prop:cat0_action}). It follows from a theorem of Bartels and L\"uck \cite{BartelsLuck2012} that $\pi_1(\mathcal{R}_K)$ satisfies the Farrell--Jones conjecture in $L$-theory. The desired rigidity statement then follows from standard consequences of surgery theory.

The asphericity of real moment-angle complexes associated to flag complexes ultimately goes back to the work of Davis on polyhedral products and reflection group constructions \cite{rafca}, while the CAT(0) geometry of the associated cubical complexes follows from Gromov's criterion for nonpositive curvature \cite{gromov} and Moussong's theorem on Coxeter groups \cite{moussong}.

We emphasize that the rigidity phenomenon described here is specific to the real case. For complex and quaternionic moment-angle complexes, the relevant pairs $(D^2,S^1)$ and $(D^4,S^3)$ fail the asphericity conditions in Davis' criterion for polyhedral products. As a result, the flag condition on $K$ does not imply asphericity in those settings, and the CAT(0) methods used here do not apply.

\paragraph*{Related rigidity results.}
Rigidity phenomena for manifolds arising from Coxeter and reflection group constructions have also been studied from an equivariant perspective. In particular, Rosas~\cite{Rosas88} proved that for right-angled reflection group actions on Coxeter manifolds, any equivariant homotopy equivalence is equivariantly homotopic to a homeomorphism. More generally, Prassidis and Spieler  \cite{prasspieler} established analogous equivariant rigidity results under suitable hypotheses on Coxeter group actions.

By contrast, recent work of Stark~\cite{stark} shows that topological rigidity fails in general for certain quotients of the Davis complex, providing counterexamples in closely related Coxeter-theoretic situations. On the positive side, Wu~\cite{wu} identified classes of quotients of the Davis complex for which topological rigidity can be recovered. Together, these results show that rigidity in the Coxeter and Davis complex setting is subtle and cannot be expected in complete generality.

The rigidity result obtained in the present paper is of a different nature and can be seen as complementary to these works. Rather than assuming the presence of a compatible group action, we establish topological rigidity in the non-equivariant sense: any homotopy equivalence between manifolds is homotopic to a homeomorphism. Our argument relies on the CAT(0) geometry of the universal cover and the Farrell--Jones conjecture, and depends only on the fundamental group.

Thus, the present work should be viewed as complementary to the equivariant rigidity theory for Coxeter group actions: while the latter uses reflection group techniques and fundamental chambers, our approach shows that for real moment-angle manifolds associated to flag complexes, rigidity is already encoded in the large-scale geometry of the fundamental group.

The main contribution of this paper is to identify a natural class of manifolds arising in toric topology for which the methods of geometric group theory and surgery theory apply. While most spaces studied in toric topology, such as quasitoric and related manifolds, are not aspherical and therefore lie outside the scope of the Borel Conjecture, we show that real moment-angle manifolds associated to flag complexes form a distinguished subclass where rigidity phenomena do occur.

Thus, the result provides a bridge between toric topology and rigidity theory, showing that real moment-angle manifolds in the flag case form a natural geometric class where the Borel Conjecture holds, in contrast to the complex and quaternionic cases where asphericity fails and such methods do not apply.

\paragraph*{Organization of the paper.}

The paper is organized as follows. In Section~\ref{sec:davis}, we recall the construction of real moment-angle manifolds via polyhedral products and establish their asphericity in the flag case. In Section~\ref{sec:cat0}, we describe the associated cubical complex and prove that its universal cover is CAT(0). In Section~\ref{sec:FJ}, we apply the Farrell--Jones conjecture and surgery theory to deduce topological rigidity.

\paragraph*{Acknowledgments.}
The author is grateful to S.~Prassidis for his guidance and mentorship during the author’s PhD studies. His encouragement to relate equivariant rigidity phenomena to the Farrell--Jones conjecture was instrumental in the development of this work. The author also thanks P.~Batakidis for helpful discussions and comments on earlier versions of this paper.

\section{Real moment--angle complexes and asphericity}
\label{sec:davis}

\subsection{Real moment--angle complexes}
\label{sec:real_MA}

Let $K$ be a simplicial complex on the vertex set $[m]=\{1,\dots,m\}$. Following \cite{tt}, we recall that the \emph{real moment--angle complex} associated to $K$ is a subspace of the cube $[-1,1]^m \subset \mathbb{R}^m$ defined as follows.

Consider the map
\[
\rho : [-1,1]^m \longrightarrow [0,1]^m,
\qquad
(u_1,\dots,u_m) \longmapsto (u_1^2,\dots,u_m^2),
\]
and let $\mathrm{cc}(K) \subset [0,1]^m$ denote the cubical complex associated to $K$. The real moment--angle complex is defined as the pullback
\[
\mathcal{R}_K := \rho^{-1}(\mathrm{cc}(K)).
\]

Equivalently, $\mathcal{R}_K$ fits into the pullback diagram
\[
\begin{tikzcd}
\mathcal{R}_K \arrow[r] \arrow[d]
&
{[-1,1]}^m \arrow[d,"\rho"]
\\
\mathrm{cc}(K) \arrow[r]
&
{[0,1]}^m.
\end{tikzcd}
\]

By construction, $\mathcal{R}_K$ is invariant under the coordinatewise action of the group $(\mathbb{Z}_2)^m = \{\pm1\}^m$ on $[-1,1]^m$, and the quotient
$\mathcal{R}_K / (\mathbb{Z}_2)^m$ is homeomorphic to $|\mathrm{cone}\,K|$.

\medskip

\paragraph*{Polyhedral product description.}

The space $\mathcal{R}_K$ admits an equivalent description as a polyhedral
product:
\[
\mathcal{R}_K = Z_K(D^1,S^0),
\]
where $D^1=[-1,1]$ and $S^0=\{\pm1\}$.

Explicitly, if $K$ is a simplicial complex on the vertex set $[m]=\{1,\dots,m\}$, then
\[
\mathcal{R}_K
=
\bigcup_{I\in K}
\left(
\prod_{i\in I} D^1
\times
\prod_{j\notin I} S^0
\right),
\]
where the union is taken over all simplices $I \subseteq [m]$ of $K$. In other words, for each simplex $I$ of $K$, we allow the coordinates indexed by $I$ to vary in $D^1=[-1,1]$, while the remaining coordinates are restricted to $S^0=\{\pm1\}$. This realizes $\mathcal{R}_K$ as a cubical subcomplex of $[-1,1]^m$.

\medskip
If $K$ is a triangulation of an $(n-1)$--dimensional sphere, then $\mathcal{R}_K$ is a closed topological manifold of dimension $n$ (see \cite[Theorem 4.1.7]{tt}). In particular, when $K$ is the boundary of the simplicial polytope dual to a simple polytope $P$, the space $\mathcal{R}_K$ is called a \emph{real moment--angle manifold}.

\begin{exs}(cf. \cite[Examples 4.1.6]{tt})

\textbf{(1) Boundary of a simplex.}
Let $K = \partial \Delta^{m-1}$ be the boundary of the standard simplex. Then \(\mathcal{R}_K \cong S^{m-1},\) the boundary of the cube $[-1,1]^m$.

\medskip

\textbf{(2) Discrete complex.}
Let $K$ consist of $m$ disjoint vertices (no edges). Then \(\mathcal{R}_K\) is the $1$--skeleton of the cube $[-1,1]^m$.

\medskip

\textbf{(3) Skeleta of the simplex.}
Let $K = \mathrm{sk}_i \Delta^{m-1}$ be the $i$--skeleton of the simplex. Then $\mathcal{R}_K$ is the $(i+1)$--skeleton of the cube $[-1,1]^m$.
\end{exs}

\subsection{Asphericity and flag complexes}
\label{sec:flag_aspherical}

The geometry of the real moment--angle complex $\mathcal{R}_K$ is governed
by the combinatorics of the simplicial complex $K$. A central role is played
by the class of flag complexes.

\begin{definition}[cf. {\cite[Definition 2.3.9]{tt}}]
A simplicial complex $K$ is called \emph{flag} if every finite set of vertices
which are pairwise connected by edges spans a simplex.
\end{definition}

Equivalently, $K$ is flag if it is determined by its $1$--skeleton, or,
equivalently, if every missing face of $K$ has cardinality two.

\medskip

\begin{exs}
\begin{enumerate}
    \item[(1)] The order complex of any poset is a flag complex.
    \item[(2)] A simple polytope $P$ is flag if and only if its nerve complex $K_P$ is flag.
    \item[(3)] The boundary of a $5$--gon is a flag complex.
    \item[(4)] The boundary of a simplex $\partial \Delta^d$ is \textit{not flag} for $d \geq 2$.
\end{enumerate}
\end{exs}

\medskip

We now recall the general asphericity criterion for polyhedral products due to Davis \cite{rafca}. Let $K$ be a simplicial complex on the vertex set $I=\{1,\dots,m\}$, and let
\[
(A,B)=\{(A(i),B(i))\}_{i\in I}
\]
be a collection of pairs of connected spaces. The associated polyhedral product $Z_K(A,B)$ is defined as a subspace of $\prod_{i\in I} A(i)$ by
\[
Z_K(A,B)
=
\bigcup_{J\in K}
\left(
\prod_{i\in J} A(i)\times \prod_{i\notin J} B(i)
\right),
\]
where the union is taken over all simplices $J\subseteq I$ of $K$. In general, the pairs $(A(i),B(i))$ are arbitrary, and the special cases of interest arise by taking $(D^n,S^{n-1})$.

\begin{thm}[Davis {\cite[Theorem 2.22]{rafca}}]
\label{thm:Davis_aspherical}
The polyhedral product $Z_K(A,B)$ is aspherical if and only if the following
conditions hold:
\begin{enumerate}
    \item[(i)] each $A(i)$ is aspherical;
    \item[(ii)] each path component of $B(i)$ is aspherical and the inclusion $\pi_1(B(i)) \to \pi_1(A(i))$ is injective;
    \item[(iii)] the simplicial complex $K$ is flag.
\end{enumerate}
\end{thm}

In the case of real moment--angle complexes, the pair $(A,B)=(D^1,S^0)$ satisfies the first two conditions: $D^1$ is contractible, and $S^0$ is discrete, hence aspherical with trivial fundamental group.

Combining Theorem~\ref{thm:Davis_aspherical} with the identification $\mathcal{R}_K = Z_K(D^1,S^0)$, we obtain the following, which can be considered a special case of \cite[Proposition 4.5]{rafca}.

\begin{cor}
\label{cor:real_MA_aspherical}
The real moment--angle complex $\mathcal{R}_K$ is aspherical if and only if $K$ is a flag complex.
\end{cor}

The rigidity results established in this paper rely crucially on the asphericity of real moment-angle manifolds. We now explain why this phenomenon is specific to the real case and fails in the complex and quaternionic settings.

We briefly recall the corresponding constructions (see \cite{tt,ho} for details). Let $K$ be a simplicial complex on $[m]$.

The \emph{complex moment-angle complex} associated to $K$ is defined as the polyhedral product
\[
\mathcal{Z}_K := Z_K(D^2,S^1),
\]
where $D^2 \subset \mathbb{C}$ is the unit disk and $S^1=\partial D^2$.

Similarly, the \emph{quaternionic moment-angle complex} is defined by
\[
\mathcal{Z}_K^{\mathbb{H}} := Z_K(D^4,S^3),
\]
where $D^4 \subset \mathbb{H}$ is the unit $4$--disk and $S^3=\partial D^4$ (\emph{quaternionic torus}).

Both constructions fit into the general framework of polyhedral products \cite{bbhg} and admit natural torus actions. However, unlike the real case, the pairs $(D^2,S^1)$ and $(D^4,S^3)$ fail to satisfy the asphericity conditions in Davis' criterion described above. This leads to the following.

\begin{lemma}
\label{lem:complex-not-aspherical}
Let $K$ be a simplicial complex with at least one non-conelike vertex\footnote{A vertex $i$ of a simplicial complex $K$ is called conelike if it is joined by an edge to every other vertex of $K$.}. Then the complex moment-angle complex \(Z_K(D^2,S^1)\) is not aspherical.
\end{lemma}

\begin{proof}
Although both $D^2$ and $S^1$ are aspherical, the inclusion
\[
S^1 \hookrightarrow D^2
\]
induces the trivial homomorphism
\[
\pi_1(S^1)\cong \mathbb{Z} \longrightarrow \pi_1(D^2)=0,
\]
which is not injective. Thus $(D^2,S^1)$ is not an aspherical pair,
and so criterion (ii) of Theorem~\ref{thm:Davis_aspherical}) fails.
\end{proof}

In the quaternionic case, we have the following result. 

\begin{lemma}
\label{lem:quaternionic-not-aspherical}
Let $K$ be a simplicial complex with at least one non-conelike vertex. Then the quaternionic moment-angle complex \(Z_K^{\mathbb{H}}(D^4,S^3)\) is not aspherical.
\end{lemma}

\begin{proof}
The boundary $S^3$ is not aspherical, since $\pi_3(S^3)\cong \mathbb{Z}$. Hence $(D^4,S^3)$ is not an aspherical pair, and so criterion (i) of Theorem~\ref{thm:Davis_aspherical} fails.
\end{proof}

\medskip

From this point on, we will assume that $K$ is a flag complex. We will see that under this assumption, the cubical structure on $\mathcal{R}_K$ satisfies Gromov's link condition, and its universal cover admits a CAT(0) metric.

\section{CAT(0) geometry and right-angled Coxeter groups}
\label{sec:cat0}

In this section we establish that the universal cover of a real moment-angle manifold associated to a flag complex admits a natural CAT(0) metric, and that its fundamental group is a CAT(0)-group in the sense of Bartels--Lück \cite{BartelsLuck2012}. This provides the geometric input required for topological rigidity.

Let $K$ be a finite simplicial complex on the vertex set $I\subseteq [m]=\{1,\dots,m\}$. Associated to the $1$-skeleton of $K$ is the right-angled Coxeter group
\[
W_K=\langle s_i \ (i\in I)\mid s_i^2=1,\ (s_is_j)^2=1 \text{ whenever } \{i,j\}\in K\rangle .
\]

Following \cite{dav, rafca}, the real moment--angle complex
\[
\mathcal{R}_K=Z_K(D^1,S^0)
\]
inherits a cubical structure from the cube $[-1,1]^m$ (also known as \emph{corner structure}),  
\[
\mathcal{R}_K=\{M_i\}_{i\in I}
\]
where \(M_i := \{ x \in \mathcal{R}_K \mid x_i = 0 \}\). Each $M_i$ is a codimension-one subcomplex, which we call \emph{mirror} in the sense of Coxeter group theory (for details see \cite[\S 1.1]{rafca}).

The associated basic construction identifies the universal cover of $\mathcal{R}_K$ with a $W_K$-space built from the canonical chamber attached to $K$, in the following way:

\paragraph*{Basic construction.} Consider the chamber
\[
K(K) := \mathcal{R}_K.
\]
The mirrors $M_i$ define a mirror structure on $K(K)$ indexed by $I$.
Following \cite{dav} and \cite[\S 1.5]{rafca}, one defines the basic construction
\[
U(W_K,K(K)) := (W_K \times K(K)) / \sim,
\]
where $(w,x) \sim (w',x)$ if $w^{-1}w'$ lies in the subgroup generated by $\{ s_i \mid x \in M_i \}$.

By construction, $W_K$ acts on $U(W_K,K(K))$ by left multiplication, and $K(K)$ is a fundamental domain for this action. Notice that if $K$ is flag, then $K$ is the nerve of $W_K$, and the chamber $K(K)$ is the standard \emph{Davis chamber} and therefore this space coincides with the \emph{Davis complex} $\Sigma(W_K,S)$ as the basic construction
\[
\Sigma_K:=\Sigma(W_K,S)=U(W_K,K(K))=(W_K\times K(K))/\sim.
\]

\begin{prop}
\label{prop:universal-cover-davis}
Let $K$ be a finite simplicial complex. Then the universal cover of $\mathcal{R}_K$ is $W_K$-equivariantly homeomorphic to the basic construction
\[
U(W_K,K(K)),
\]
where $K(K)$ denotes the canonical cubical chamber associated to $K$.
\end{prop}

\begin{proof}
We construct a covering map
\[
\pi : U(W_K,K(K)) \longrightarrow \mathcal{R}_K.
\]
Define $\pi([w,x]) = w \cdot x$, where $w$ acts via the reflections described above. This is well-defined because if $(w,x) \sim (w',x)$, then $w^{-1}w'$ fixes $x$ by definition of the equivalence relation.

The map $\pi$ is continuous, surjective, and locally a homeomorphism, since locally it corresponds to gluing chambers along mirrors. Moreover, $U(W_K,K(K))$ is connected and simply connected: it is obtained by gluing contractible chambers along codimension-one subspaces in the standard Coxeter manner, and this construction yields a simply connected space (see \cite[Example 2.2]{rafca} and references therein). Thus $\pi$ is the universal covering map, and
\[
\widetilde{\mathcal{R}_K} \cong U(W_K,K(K)).
\]
\end{proof}

\medskip

We now describe the curvature properties of this space.

\begin{prop}
\label{prop:cat0_action}
Let $K$ be a finite flag simplicial complex such that $\mathcal{R}_K$ is a closed manifold. Then the group
\[
G=\pi_1(\mathcal{R}_K)
\]
acts properly, cocompactly, and isometrically on a finite-dimensional CAT(0) space.
\end{prop}
\begin{proof}
By Proposition~\ref{prop:universal-cover-davis}, the universal cover $\widetilde{\mathcal{R}_K}$ is identified with the Davis complex $\Sigma_K$.

In the right-angled case, $\Sigma_K$ admits a natural piecewise Euclidean cubical metric in which each cube is isometric to a standard Euclidean cube. The link of each vertex is isomorphic to the simplicial complex $K$.

If $K$ is flag, then every vertex link is a flag simplicial complex. Hence, by Gromov's link condition, $\Sigma_K$ is nonpositively curved. Since $\Sigma_K$ is simply connected, it follows that it is a CAT(0) space.

The action of $G=\pi_1(\mathcal{R}_K)$ on $\widetilde{\mathcal{R}_K} \cong \Sigma_K$ is the deck transformation action. Because $\mathcal{R}_K$ is compact, this action is properly discontinuous and cocompact. Moreover, the cubical metric is invariant under the Coxeter group action and hence under its subgroup $G$, so the action is by isometries.

Finally, $\Sigma_K$ is finite-dimensional since $K$ is finite.
\end{proof}

\begin{remark}[General Coxeter Case]
Moussong's theorem \cite{moussong} shows that the Davis complex of any Coxeter system admits a CAT(0) metric. However, in the general (non-right-angled) case, this metric is polyhedral rather than cubical, and the CAT(0) property is not determined purely by the combinatorics of the nerve.

In contrast, in the right-angled case, the cubical structure allows one to characterize nonpositive curvature entirely in terms of the flag condition on $K$. This combinatorial control is essential for the connection with real moment-angle manifolds.
\end{remark}

\section{Farrell--Jones and topological rigidity}
\label{sec:FJ}

In this section we combine the CAT(0) geometry established above with results of Bartels and L\"uck \cite{BartelsLuck2012} to deduce topological rigidity of real moment--angle manifolds.

\medskip

Let $K$ be a finite flag simplicial complex such that
\[
\mathcal{R}_K = Z_K(D^1,S^0)
\]
is a closed manifold of dimension $n \geq 5$, and let
\[
G = \pi_1(\mathcal{R}_K).
\]

We briefly recall that the Farrell--Jones conjecture in $L$-theory predicts that for a group $G$, the assembly map
\[
H_n(BG;\mathbf{L}^{\langle -\infty \rangle})
\longrightarrow
L_n^{\langle -\infty \rangle}(\mathbb{Z}G)
\]
is an isomorphism for all $n$ (see \cite{lr05}).

\begin{lemma}
\label{lem:FJ_CAT0}
The group $G$ satisfies the Farrell--Jones conjecture in $L$-theory.
\end{lemma}

\begin{proof}
By Proposition~\ref{prop:cat0_action}, the group $G$ acts properly, cocompactly, and isometrically on the finite-dimensional CAT(0) space $\widetilde{\mathcal{R}_K}$. Hence $G$ satisfies the Farrell--Jones conjecture by Bartels--L\"uck \cite{BartelsLuck2012}.
\end{proof}

\medskip

The following result is a standard consequence of the Farrell--Jones conjecture in surgery theory. Recall that the (topological) structure set $\mathcal{S}(M)$ of a closed manifold $M$ consists of equivalence classes of homotopy equivalences $f \colon N \to M$, where $N$ is a closed manifold, with two such maps considered equivalent if they differ by a homeomorphism of $N$ homotopic to the identity. In particular, $\mathcal{S}(M)$ measures the failure of homotopy equivalence to imply homeomorphism.

It is a standard consequence of surgery theory that $\mathcal{S}(M)$ fits into the surgery exact sequence, and that the vanishing of the relevant surgery obstruction groups implies that $\mathcal{S}(M)$ is trivial (see, e.g., \cite{wall70, lm24}).

\begin{lemma}
\label{lem:FJ_Borel}
Let $M$ be a closed aspherical manifold of dimension $n \geq 5$ with fundamental group $G$. If $G$ satisfies the Farrell--Jones conjecture in $L$-theory, then the structure set $\mathcal{S}(M)$ is trivial. In particular, $M$ satisfies the Borel Conjecture.
\end{lemma}

\begin{proof}
By the Farrell--Jones conjecture, the $L$-theoretic assembly map
\[
H_n(BG; \mathbf{L}^{\langle -\infty \rangle})
\longrightarrow
L_n^{\langle -\infty \rangle}(\mathbb{Z}G)
\]
is an isomorphism. Since $M$ is aspherical, $M \simeq BG$, and the surgery exact sequence identifies the structure set $\mathcal{S}(M)$ with the kernel of the surgery obstruction map.

The assembly isomorphism implies that the relevant obstruction groups are completely determined by homology, and in particular that the structure set reduces to a single element. Hence any homotopy equivalence $f : N \to M$ is homotopic to a homeomorphism.

See \cite{BartelsLuck2012} for details.
\end{proof}

\medskip

We can now state the main result.

\begin{thm}
\label{thm:Borel_main}
Let $K$ be a finite flag simplicial complex such that $\mathcal{R}_K$ is a closed manifold of dimension $n \geq 5$. Then $\mathcal{R}_K$ is aspherical and satisfies the Borel Conjecture.
\end{thm}

\begin{proof}
By Corollary~\ref{cor:real_MA_aspherical}, the flag condition on $K$ implies that $\mathcal{R}_K$ is aspherical. By Lemma~\ref{lem:FJ_CAT0}, the group $G=\pi_1(\mathcal{R}_K)$ satisfies the Farrell--Jones conjecture. The result now follows from Lemma~\ref{lem:FJ_Borel}.
\end{proof}

\begin{remark}
The rigidity established in this paper differs from the equivariant rigidity results for Coxeter group actions due to Prassidis and Spieler \cite{prasspieler} (compare also with \cite{Rosas88}). In their setting, one assumes the presence of a Coxeter group $W$ acting on manifolds $M$ and $M'$, and shows that a $W$-equivariant homotopy equivalence $f : M \to M'$ is homotopic to a $W$-equivariant homeomorphism.

In contrast, our result is non-equivariant: no group action is assumed on the target manifold, and rigidity is deduced purely from the properties of the fundamental group. The proof relies on the CAT(0) structure of the universal cover and the validity of the Farrell--Jones conjecture for the corresponding group.

These two approaches highlight different mechanisms underlying rigidity. The equivariant theory depends on the combinatorial and geometric structure of reflection group actions, whereas the present approach shows that, for real moment-angle manifolds, topological rigidity is a consequence of the large-scale geometry of the fundamental group.
\end{remark}

\paragraph*{Further discussion.}
The argument presented in this paper suggests a possible extension beyond real moment-angle manifolds. More generally, one may consider polyhedral products $Z_K(A,B)$ satisfying the asphericity conditions of Davis. In this setting, asphericity alone is not sufficient to ensure rigidity: additional geometric structure on the universal cover is required.

In the real case, the pair $(D^1,S^0)$ gives rise to a natural cubical structure whose universal cover can be identified with the Davis complex, leading to a CAT(0) metric and a proper cocompact group action. For general pairs $(A,B)$, such a structure need not be available, and the connection with CAT(0) geometry becomes more subtle.

It is therefore natural to ask \emph{which aspherical polyhedral products admit geometric structures strong enough to place their fundamental groups within the scope of the Farrell--Jones conjecture}. This suggests a broader problem of identifying classes of polyhedral products for which rigidity phenomena analogous to those established here hold.

\end{document}